\documentclass[11pt]{amsart}
\usepackage{amsmath,amssymb}
\newtheorem{theorem}{Theorem}[section]

\newtheorem{proposition}{Proposition}[section]
\newtheorem{corollary}{Corollary}[section]
\newtheorem{lemma}{Lemma}[section]
\theoremstyle{definition}
\newtheorem{definition}{Definition}[section]
\newtheorem{remark}{Remark}[section]
\newtheorem{observation}{Observation}[section]

\usepackage{latexsym,amssymb}
\begin{document}
\baselineskip=14.5pt
\title[Generalization of some weighted zero-sum theorems]{Generalization of some weighted zero-sum theorems and related Extremal sequence} 

\author{Subha Sarkar}
\address {Department of Mathematics, Birla Institute of Technology Mesra-835 215 (Ranchi), Jharkhand, India}
\email[Subha Sarkar]{subhasarkar2051993@gmail.com } 

\begin{abstract}
Let $G$ be a finite abelian group of exponent $n$ and let $A$ be a non-empty subset of $[1,n-1]$. The Davenport constant of $G$ with weight $A$, denoted by $D_A(G)$, is defined to be the least positive integer $\ell$ such that any sequence over $G$ of length $\ell$ has a non-empty $A$-weighted zero-sum subsequence. Similarly, the combinatorial invariant $E_{A}(G)$ is defined to be the least positive integer $\ell$ such that any sequence over $G$ of length $\ell$ has an $A$-weighted zero-sum subsequence of length $|G|$. In this article, we determine the exact value of $D_A(\mathbb{Z}_n)$, for some particular values of $n$, where $A$ is the set of all cubes in $\mathbb{Z}_n^*$. We also determine the structure of the related extremal sequence in this case.
\end{abstract}

\maketitle

\section{Introduction}
Let $G$ be a finite abelian group (written additively) and let exp($G$) be the {\it exponent} of the group $G$. We denote the free abelian monoid with basis $G$ by $\mathcal{F}(G)$. That is, every element $S \in \mathcal{F}(G)$ has a unique representation of the form $$S = \prod_{g \in G} g^{\mathsf{v}_g(S)} \qquad \mbox{ with } \mathsf{v}_g(S) \in  \mathbb{N}_0.$$ We call an element $S$ of $\mathcal{F}(G)$ to be a {\it sequence over $G$}.

\medskip

\begin{definition}
Let $S$ and $T$ be two sequences over $G$. Then $T$ is said to be a {\it subsequence of $S$} if $\mathsf{v}_g(T) \leq \mathsf{v}_g(S)$ for all $g \in G$. We denote a subsequence $T$ of $S$ by $T \mid S$. 
\end{definition}

\medskip

\begin{definition}
A sequence $S=g_1 \cdot g_2 \cdot \ldots \cdot g_t$ over $G$ is said to be {\it a zero-sum sequence} if $\sigma(S)=\displaystyle \sum_{i=1}^t g_i= 0$, where $0$ is the identity element of $G$.
\end{definition}


\begin{definition}
For a finite abelian group $G$, the {\it Davenport constant} $D(G)$ is defined to be the least positive integer $\ell$ such that any sequence over $G$ of length $\ell$ has a non-empty zero-sum subsequence.
\end{definition}

\begin{definition}
For a finite abelian group $G$, the constant $\mathsf{E}(G)$ is defined to be the least positive integer $\ell$ such that any sequence over $G$ of length $\ell$ has a zero-sum subsequence of length $\vert G \vert$.
\end{definition}

Gao \cite{gao} proved that, for a finite abelian group $G$ the constants $D(G)$ and $\mathsf{E}(G)$ are related by the relation
\begin{equation} \label{2eq1}
\mathsf{E}(G) = D(G)+|G|-1.
\end{equation}

The generalization of these constants with weights has been considered in \cite{abpr,ac,acfkp,th}. First we define the following.

\begin{definition}
Let $G$ be a finite abelian group of exponent $n$ and let $A$ be a non-empty subset of $[1,n-1]$. A sequence $S = g_1 \cdot g_2 \cdot \ldots \cdot g_k$ over $G$ is said to be an {\it $A$-weighted zero-sum sequence} if there exist $a_1, \ldots , a_k$ in $A$ such that $\displaystyle \sum_{i=1}^k a_i g_i = 0$. 
\end{definition}

For a finite abelian group $G$ of exponent $n$ and a non-empty subset $A$ of $[1,n-1]$, we define the weighted Davenport constant $D_A(G)$ and the constant $\mathsf{E}_A(G)$ as follows.

\begin{definition}
The {\it weighted Davenport constant of $G$ with weight $A$}, denoted by $D_A(G)$, is defined to be the least positive integer $\ell$ such that any sequence over $G$ of length $\ell$ has a non-empty $A$-weighted zero-sum subsequence.
\end{definition}

\begin{definition}
The constant $\mathsf{E}_A(G)$ is defined to be the least  positive integer $\ell$ such that any sequence over $G$ of length $\ell$ has an $A$-weighted zero-sum subsequence of length $|G|$.
\end{definition}

When $A=\{1\}$, the constants $D_A(G)$ and $\mathsf{E}_A(G)$ are the constants $D(G)$ and $\mathsf{E}(G)$ respectively.
Grynkiewicz, Marchan and Ordaz \cite{gmo} proved that the weighted generalization of Gao's relation \eqref{2eq1} namely, 
\begin{equation}\label{2eq2}
\mathsf{E}_A(G) = D_A(G)+|G|-1
\end{equation}
holds true for a general finite abelian group $G$ and weight set $A$.

We fix some notations first. For a positive integer $n$, we denote the number of prime factors of $n$ counted with multiplicity (respectively, without multiplicity) by $\Omega(n)$ (respectively, by $\omega(n$)). We also denote $\mathbb{Z}_n$ to denote the cyclic group $\mathbb{Z}/n\mathbb{Z}$ and use $\mathbb{Z}_n^*$ to denote the group of units of $\mathbb{Z}_n$.

In \cite{acfkp}, it was proved that $\mathsf{E}_A(\mathbb{Z}_n) \geq n+\Omega(n)$. They also conjectured that it is the exact value. This conjecture has been proved in \cite{gr} and \cite{luca} independently.



%
%
%
%
%
%
%
%
%
Adhikari, David and Urroz \cite{adu} considered the problem of determining the values of $D_{R_n}(\mathbb{Z}_n)$ and $E_{R_n}(\mathbb{Z}_n)$ for squarefree $n$, where $R_n = \{a^2 \pmod n : a \in \mathbb{Z}_n^* \}$. The case when $n$ is a prime was already considered by Adhikari and Rath in \cite{ar}. Chintamani and Moriya \cite{CM} extended some results of Adhikari, David and Urroz for non square-free integer $n$.

In \cite{subha}, we proved an upper bound of $D_{T_n}(\mathbb{Z}_n)$ and $\mathsf{E}_{T_n}(\mathbb{Z}_n)$ for the weight set $T_n = \{a^3 \pmod n \;|\; a \in \mathbb{Z}_n^*\}$ to be the set of all cubes in $\mathbb{Z}_n^*$. More precisely, we proved the following:

{\it If $n=7^ln_1n_2$ is an odd integer such that $n_1 =\prod_{i =1}^r p_i^{e_i}$ and $n_2=\prod_{j =1}^s q_j^{f_j}$  with primes $p_i \equiv 1 \pmod 3$ and $q_j \equiv 2 \pmod 3$ and $7 \nmid n_1$, then 
\begin{itemize}
\item[(i)] $D_{T_n}(\mathbb{Z}_n) \leq 3\Omega(n_1)+\Omega(n_2)+5l+1$, and 
\smallskip
\item[(ii)] $E_{T_n}(\mathbb{Z}_n) \leq n+3\Omega(n_1)+\Omega(n_2)+5l$.
\end{itemize}}

In this article, one of our results is to prove the exact value of $D_{T_n}(\mathbb{Z}_n)$ when $n$ is a square-free odd integer which is not divisible by 7 and 13. More precisely, we prove the following theorem.

\begin{theorem}\label{mt1}
Let $n=n_1n_2$ be a square-free odd integer which is not divisible by 7 and 13 such that $n_1 =\prod_{i =1}^r p_i$ and $n_2=\prod_{j =1}^s q_j$  with primes $p_i \equiv 1 \pmod 3$ and $q_j \equiv 2 \pmod 3$. Then we have
\begin{itemize}
\item[(i)] $D_{T_n}(\mathbb{Z}_n) = 2\Omega(n_1)+ \Omega(n_2)+1$, and 
\smallskip
\item[(ii)] $E_{T_n}(\mathbb{Z}_n) = n+2\Omega(n_1)+ \Omega(n_2)$.
\end{itemize}
\end{theorem}

\medspace

For a non-empty subset $A$ of $[1,n-1]$, it is clear from the definition of $D_A(\mathbb{Z}_n)$ that, there is a sequence over $\mathbb{Z}_n$ of length $D_{A}(\mathbb{Z}_n)-1$ which does not have any non-empty $A$-weighted zero-sum subsequence. We call such sequence an $A$-extremal sequence.

\medskip

\begin{definition} \label{equi}
Let $A$ be a subgroup of $\mathbb{Z}_n^*$. We say two sequences $x_1 \cdot \ldots \cdot x_k$ and $y_1 \cdot \ldots \cdot y_k$ over $\mathbb{Z}_n$ are {\it equivalent with respect to $A$} if there are $a_1, a_2, \ldots, a_k \in A$, $c \in \mathbb{Z}_n^*$ and a permutation $\sigma \in S_k$ such that $y_i=c a_i x_{\sigma(i)}$ for all $1\leq i \leq k$.
\end{definition}

It is clear from Definition \ref{equi} that, if $x_1 \cdot \ldots \cdot x_k$ and $y_1 \cdot \ldots \cdot y_k$ are equivalent with respect to $A$, then $x_1 \cdot \ldots \cdot x_k$ does not have any non-empty A-weighted zero-sum subsequence if and only if $y_1 \cdot \ldots \cdot y_k$ does not have any non-empty A-weighted zero-sum
subsequence. Thus when the weight set $A$ is a subgroup of $\mathbb{Z}_n^*$, we want to characterize the $A$-extremal sequences up to this equivalence.

\medskip

The structure of $A$-extremal sequences for the weight sets $A= \mathbb{Z}_n^*$ and $A=R_n$ has been studied in \cite{amp} and in \cite{ahms} respectively. Here, for a square-free odd integer $n$ which is not divisible by 7 and 13, we characterize the $T_n$-extremal sequence as follows. 

\begin{theorem}\label{extremal}
Let $n=n_1n_2$ be a square-free odd integer which is not divisible by 7 and 13 such that $n_1 =\prod_{i =1}^r p_i$ and $n_2=\prod_{j =1}^s q_j$  with primes $p_i \equiv 1 \pmod 3$ and $q_j \equiv 2 \pmod 3$. Let $S: x_1 \cdot x_2 \cdot \ldots \cdot x_{\ell}$ be a sequence over $\mathbb{Z}_n$ of length $\ell=D_{T_n}(\mathbb{Z}/n\mathbb{Z})-1=2\Omega(n_1)+ \Omega(n_2)$ such that $S$ does not have any non-empty $T_n$-weighted zero-sum subsequence. Then $S$ is equivalent to one of the following sequences:

{\noindent \sf Case 1.} 

There exists a prime divisor $p$ of $n_1$ such that $p$ divides all the terms of $S$, except two, say $x_1$ and $x_2$.

Let $n^{'}=\frac{n}{p}$ and $x_i=py_i$ for all $3\leq i\leq \ell$. Then the sequence $S_1^{'}:y_3 \cdot y_4 \cdot \ldots \cdot y_{\ell}$ over $\mathbb{Z}_{n^{'}}$ is an $T_{n^{'}}$-extremal sequence in $\mathbb{Z}_{n^{'}}$, and the image of the sequence $x_1 \cdot x_2$ under the natural map $\mathbb{Z}_n \rightarrow \mathbb{Z}_p$ is $T_p$-extremal.

\medspace

{\noindent \sf Case 2.}

There exists a prime divisor $p$ of $n_2$ such that $p$ divides all the terms of $S$, except one, say $x_1$.

Let $n^{'}=\frac{n}{p}$ and $x_i=py_i$ for all $2\leq i\leq \ell$. Then the sequence $S_1^{'}:y_2 \cdot y_3 \cdot \ldots \cdot y_{\ell}$ over $\mathbb{Z}_{n^{'}}$ is an $T_{n^{'}}$-extremal sequence in $\mathbb{Z}_{n^{'}}$.
\end{theorem}

%
%
%
%
%

\section{Preliminaries}

Let $p$ be an odd prime number such that $p \equiv 2 \pmod 3$. Then every element of $\mathbb{Z}_{p^r}^*$ is a cubic residue modulo $p^r$. Thus in this case, $T_{p^r} = \mathbb{Z}_{p^r}^*$. 

Let $p$ be a prime number such that $p \equiv 1 \pmod 3$ and let $g$ be a generator of the cyclic group $\mathbb{Z}_{p^r}^*$. Then $T_{p^r} = \{ g^3,g^{2.3},\ldots, g^{\frac{p^r-p^{r-1}}{3}.3}=1\}$ forms a subgroup of $\mathbb{Z}_{p^r}^*$ of order $\frac{p^r-p^{r-1}}{3}$.

\medspace 

The following lemma is a consequence of the Chinese Remainder Theorem.

\medspace

\begin{lemma} \cite{ir} \label{crt}
Suppose that $m = 2^e p_1^{e_1}\ldots p_{\ell}^{e_{\ell}}$. Then $x^n\equiv a \pmod m$ is solvable if and only if the system of congruences $x^n\equiv a \pmod {2^e}$, $x^n\equiv a \pmod {p_1^{e_1}}$, \ldots, $x^n\equiv a \pmod {p_{\ell}^{e_{\ell}}}$ is solvable.
\end{lemma}

Let $n$ be a positive integer. We denote $p^r \mid \mid n$ to mean that $p^r \mid n$ but $p^{r+1} \nmid n$. Let $S$ be a sequence over $\mathbb{Z}_n$ and $p$ a prime such that $p^r \mid \mid n$. We denote $S^{(p)}$ to be the image of the sequence $S$ under the natural map $\mathbb{Z}_n \rightarrow \mathbb{Z}_{p^r}$. Then we have the following trivial observation from Chinese Remainder Theorem and Lemma \ref{crt}.

\begin{observation}\label{obs}
Let $n$ be a positive integer and $S$ a sequence over $\mathbb{Z}_n$. Then $S$ is an $T_n$-weighted zero-sum sequence if and only if $S^{(p)}$ is an $T_{p^r}$-weighted zero-sum sequence, for every prime divisor $p$ of $n$ such that $p^r \mid \mid n$.	
\end{observation}

We now state some lemmas which are useful to prove our main theorems.

\begin{lemma} \label{prime} \cite{shameek}
Let $p$ be an odd prime such that $p \equiv 1 \pmod 3$. Then $D_{T_p}(\mathbb{Z}_p)=3$.
\end{lemma}

\medspace

\begin{lemma}\label{three} \cite{shameek}
	Let $p$ be an odd prime such that $p \neq 7, 13$. Let $S = x_1\cdot x_2\cdot \ldots \cdot x_m$ be a sequence  over $\mathbb{Z}_p$ such that at least three terms of $S$ are units in $\mathbb{Z}_p$. Then there exist $a_1,a_2, \ldots,a_m \in T_p$ such that $\sum_{i=1}^m a_ix_i \equiv 0 \pmod p$.
\end{lemma}

\medspace

\begin{lemma}\label{two} \cite{gr}
	Let $p^r$ be an odd prime power and $S = x_1 \cdot x_2\cdot \ldots \cdot x_m$ a sequence over $\mathbb{Z}_{p^r}$ such that at least two elements of $S$ are units in $\mathbb{Z}_{p^r}$. Then there exist $a_1,a_2,\ldots, a_m \in \mathbb{Z}_{p^r}^*$ such that $\sum_{i=1}^m a_ix_i \equiv 0 \pmod {p^r}$.
\end{lemma}

We also use the following remark which was proved in \cite{CM}.

\begin{remark}\label{2rmk2}
Let $m$ and $n$ be two positive integers such that $m$ divides $n$ and $\psi : \mathbb{Z}_n \rightarrow \mathbb{Z}_m$ be the natural ring homomorphism. Then $\mathbb{Z}_m^* = \psi(\mathbb{Z}_n^*)$ and hence $T_m = \psi(T_n)$.
\end{remark}

\section{Proof of Theorem \ref{mt1}}
First, we prove the following lemma for a general weight set $A$, which is useful to prove the lower bound.

\medskip

\begin{lemma} \label{lowerbound}
Let $n=n_1n_2$ and $A,A_1,A_2$ be subsets of $\mathbb{Z}_n,\mathbb{Z}_{n_1},\mathbb{Z}_{n_2}$ respectively such that the image of $A$ under the natural map $\mathbb{Z}_n \rightarrow \mathbb{Z}_{n_i}$ is contained in $A_i$ for $i=1,2$. Then $D_A(\mathbb{Z}_n) \geq D_{A_1}(\mathbb{Z}_{n_1})+D_{A_2}(\mathbb{Z}_{n_2})-1$.
\end{lemma}

\begin{proof}
Let $D_{A_1}(\mathbb{Z}_{n_1})=k$ and $D_{A_1}(\mathbb{Z}_{n_2})=\ell$. Then there exists a sequence $S_1^{'}=x_1^{'} \cdot x_2^{'} \cdot \ldots \cdot x_{k-1}^{'}$ over $\mathbb{Z}_{n_1}$ of length $k-1$ which does not have any non-empty $A_1$-weighted zero-sum subsequence. Similarly there exists a sequence $S_2^{'}=y_1^{'} \cdot y_2^{'} \cdot \ldots \cdot y_{\ell-1}^{'}$ over $\mathbb{Z}_{n_2}$ of length $\ell-1$ which does not have any non-empty $A_2$-weighted zero-sum subsequence.

For $1 \leq i \leq k-1$, let $x_i=n_2 x_i^{'}$ and  $S_1=x_1 \cdot x_2 \cdot \ldots \cdot x_{k-1}$. Similarly for $1 \leq i \leq \ell-1$, let $y_i \in \mathbb{Z}_n$ be such that $y_i \mapsto y_i^{'}$ under the natural ring homomorphism $\mathbb{Z}_n \rightarrow \mathbb{Z}_{n_2}$, and $S_2=y_1 \cdot y_2 \cdot \ldots \cdot y_{\ell-1}$.

Let us consider the sequence $S=x_1 \cdot x_2 \cdot \ldots \cdot x_{k-1}\cdot y_1 \cdot y_2 \cdot \ldots \cdot y_{\ell-1}$ over $\mathbb{Z}_n$ of length $k+\ell-2$. We show that $S$ does not have any non empty $A$-weighted zero-sum subsequence.

If possible let $S$ has a non empty $A$-weighted zero-sum subsequence, say $T$. If $T$ is a proper subsequence of $S_1$, i.e. if $T=x_{i_1}\cdot x_{i_2}\cdot \ldots \cdot x_{i_t}$ does not contain any term from $S_2$, then there exists $a_{i_1}, a_{i_2}, \ldots, a_{i_t} \in A$ such that $$a_{i_1}x_{i_1}+ a_{i_2}x_{i_2}+ \cdots+ a_{i_t}x_{i_t} \equiv 0 \pmod n.$$
As the image of $A$ under the natural map $\mathbb{Z}_n \rightarrow \mathbb{Z}_{n_1}$ is contained in $A_1$, by dividing the above congruence by $n_2$ we get an $A_1$-weighted zero-sum subsequence of $S_1^{'}$. This is not possible by our choice of $S_1^{'}$.

Thus we can assume that $T$ contains some term from $S_2$. By considering the $A$-weighted zero-sum combination of $T$, we get an $A$-weighted combination of some terms of $S_2$ which is divisible by $n_2$. As the image of $A$ under the natural map $\mathbb{Z}_n \rightarrow \mathbb{Z}_{n_2}$ is contained in $A_2$, we get that the sequence $S_2^{'}$ has an $A_2$-weighted zero-sum subsequence. A contradiction to our choice of $S_2^{'}$. 

Hence we see that $S$ does not have any non empty $A$-weighted zero-sum subsequence, and therefore $D_A(\mathbb{Z}_n) \geq k+\ell-1=D_{A_1}(\mathbb{Z}_{n_1})+D_{A_2}(\mathbb{Z}_{n_2})-1$.
\end{proof}

\medskip

\begin{proposition}\label{lowerbound2}
Let $n=n_1n_2$ be an odd integer such that $n_1 =\prod_{i =1}^r p_i^{e_i}$ and $n_2=\prod_{j =1}^s q_j^{f_j}$  with primes $p_i \equiv 1 \pmod 3$ and $q_j \equiv 2 \pmod 3$.   Then $D_{T_n}(\mathbb{Z}_n)\geq 2 \Omega(n_1)+\Omega(n_2)+1$, and $E_{T_n}(\mathbb{Z}_n) \geq n+2 \Omega(n_1)+\Omega(n_2)$.
\end{proposition}

\begin{proof}
Let $p$ be a prime such that $p \equiv 1 \pmod 3$, then from Lemma \ref{prime}, we get $D_{T_p}(\mathbb{Z}_p)=3=2 \Omega(p)+1$. Again if $p$ is a prime such that $p \equiv 2 \pmod 3$, then we know that $T_p=\mathbb{Z}_p^*$, and hence $D_{T_p}(\mathbb{Z}_p)=2= \Omega(p)+1$.

Since $\displaystyle\Omega(n_1)=\Omega \left(\prod_{i =1}^r p_i^{e_i}\right)=\sum_{i =1}^r \Omega(p_i^{e_i})$, by Remark \ref{2rmk2} and Lemma \ref{lowerbound} we get, $D_{T_{n_1}}(\mathbb{Z}_{n_1}) \geq 2 \Omega(n_1)+1$.

Similarly, as $\displaystyle\Omega(n_2)=\Omega\left(\prod_{j =1}^s q_j^{f_j}\right)=\sum_{j =1}^s \Omega(q_j^{f_j})$, by Remark \ref{2rmk2} and Lemma \ref{lowerbound} we get, $D_{T_{n_2}}(\mathbb{Z}_{n_2}) \geq  \Omega(n_2)+1$.

Therefore again by Lemma \ref{lowerbound}, we get,
\begin{eqnarray*}
D_{T_n}(\mathbb{Z}_n) &\geq & D_{T_{n_1}}(\mathbb{Z}_{n_1})+D_{T_{n_2}}(\mathbb{Z}_{n_2})-1\\
&\geq &  2 \Omega(n_1)+1+\Omega(n_2)+1-1 \\
&=&  2 \Omega(n_1)+\Omega(n_2)+1,
\end{eqnarray*}
and this proves the lower bound for $D_{T_n}(\mathbb{Z}_n)$. Hence from the relation \eqref{2eq2}, we get $E_{T_n}(\mathbb{Z}_n) \geq n+2 \Omega(n_1)+\Omega(n_2)$.
\end{proof}

To prove the upper bound, we need the following proposition. We use Lemma \ref{three}, Lemma \ref{two} and the Chinese Remainder Theorem to prove the proposition.

\medskip

\begin{proposition}\label{2prop1}
Let $n=n_1n_2$ be a square-free odd integer which is not divisible by 7 and 13 such that $n_1 =\prod_{i =1}^r p_i$ and $n_2=\prod_{j =1}^s q_j$  with primes $p_i \equiv 1 \pmod 3$ and $q_j \equiv 2 \pmod 3$. Let $m \geq 3\omega(n_1)+2\omega(n_2)$ and $S = x_1 \cdot x_2\cdot \ldots \cdot x_{m+2\Omega(n_1)+\Omega(n_2)}$ be a sequence over $\mathbb{Z}_n$. Then there exists a subsequence $x_{i_1}\cdot x_{i_2}\cdot\ldots \cdot x_{i_m}$ of $S$ and $a_1,a_2,\ldots,a_m \in T_n$ such that $\sum_{j=1}^m a_j x_{i_j} \equiv 0 \pmod n$.
\end{proposition}

\begin{proof}
We proceed by induction on $\Omega(n)$. When $\Omega(n) = 1$, then $n$ is a prime, say $n = p$. If $p \equiv 1 \pmod 3$ then by our assumption, $p > 13$. By Lemma \ref{three}, if a sequence $S = x_1\cdot x_2\cdot \ldots \cdot x_{m+2}$ over $\mathbb{Z}_p$ has at least three non-zero elements, then there are $a_i \in T_p$ for $i = 1,2,\ldots, m$ such that $\sum_{i =1}^m a_i x_i \equiv 0 \pmod p$. Otherwise there is a subsequence $S_1 = x_{i_1}\cdot x_{i_2} \cdot \ldots \cdot x_{i_m}$ of $S$ of length $m$ such that all elements of $S_1$ are divisible by $p$. Hence $\sum_{j=1}^m a_j x_{i_j} \equiv 0 \pmod p$ for any choice of $a_i \in T_p$.

Now if $p \equiv 2 \pmod 3$, then we know that $T_p =\mathbb{Z}_p^*$. Hence by Lemma \ref{two}, if a sequence $S = x_1\cdot x_2 \cdot \ldots \cdot x_{m+1}$ over $\mathbb{Z}_p$ has at least two non-zero elements, then there are $a_i \in T_p$ for $i = 1,2,\ldots, m$ such that $\sum_{i =1}^m a_i x_i \equiv 0 \pmod p$. Otherwise there is a subsequence $S_1 = x_{i_1}\cdot x_{i_2}\cdot \ldots \cdot x_{i_m}$ of $S$ of length $m$ such that all elements of $S_1$ are divisible by $p$. Hence $\sum_{j=1}^m a_j x_{i_j} \equiv 0 \pmod p$ for any choice of $a_i \in T_p$. This proves the proposition when $\Omega(n) = 1$.

Suppose now that $\Omega(n) \geq 2$ and the result is true for any odd integer $N = N_1N_2$ such that $7,13 \nmid N$ where each prime divisor of $N_1$ is congruent to 1 modulo 3 and each prime divisor of $N_2$ is congruent to 2 modulo 3 with $\Omega(N_1) < \Omega(n_1)$ or $\Omega(N_2) < \Omega(n_2)$. Let $S = x_1\cdot x_2\cdot \ldots \cdot x_{m+2\Omega(n_1)+\Omega(n_2)}$ be a sequence over $\mathbb{Z}_n$.

\medskip

{\noindent \sf Case 1.} There exists a prime $p_t \mid n_1$ such that the sequence $S$ contains at most two elements which are co prime to $p_t$.

In this case, we remove those elements and consider the subsequence $S_1$ of $S$ of length at least $m + 2\Omega(\frac{n_1}{p_t})+\Omega(n_2)$ all of whose elements are zero modulo $p_t$. Since $m \geq 3 \omega(\frac{n_1}{p_t})+2\omega(n_2)$, by the induction hypothesis, there is a subsequence $x_{i_1}\cdot x_{i_2}\cdot \ldots \cdot x_{i_m}$ of $S_1$ and $a_1,a_2,\ldots,a_m \in T_{\frac{n}{p_t}}$ such that $$\sum_{j=1}^m a_j \frac{x_{i_j}}{p_t} \equiv 0 \pmod {n/p_t}.$$ Since $n^{'} = n/p_t$ divides $n$, by Remark \ref{2rmk2}, we see that $\psi(T_n) = T_{\frac{n}{p_t}}$. Thus for each $a_j$ there exists $a_j^{'} \in T_n$ such that $\psi(a_j^{'}) = a_j$ i.e. $a_j^{'} \equiv a_j \pmod {n/p_t}$. Hence $$\sum_{j=1}^m a_j^{'} \frac{x_{i_j}}{p_t} \equiv 0 \pmod {n/p_t}.$$ Therefore, $$\sum_{j=1}^m a_j^{'} x_{i_j} \equiv 0 \pmod n.$$

\medskip

{\noindent \sf Case 2.} There is a prime $q_t \mid n_2$ such that the sequence $S$ contains at most one element which is co prime to $q_t$.

In this case, we remove this element and consider the subsequence $S_1$ of $S$ of length at least $m +2\Omega(n_1) +\Omega(\frac{n_2}{q_t})$ all of whose elements are zero modulo $q_t$. Since $m \geq 3\omega(n_1)+ 2 \omega(\frac{n_2}{q_t})$, by the induction hypothesis, there is a subsequence $x_{i_1}\cdot x_{i_2}\cdot \ldots \cdot x_{i_m}$ of $S_1$ and $a_1,a_2,\ldots,a_m \in T_{\frac{n}{q_t}}$ such that $$\sum_{j=1}^m a_j \frac{x_{i_j}}{q_t} \equiv 0 \pmod {n/q_t}.$$ As in the Case 1, using Remark \ref{2rmk2}, there exist $a_1^{'}, a_2^{'},\ldots, a_m^{'} \in T_n$ such that $\sum_{j=1}^m a_j^{'} x_{i_j} \equiv 0 \pmod n$.

\medskip

{\noindent \sf Case 3.} For all primes $p_i \mid n_1$, the sequence $S$ contains at least three unit elements modulo $p_i$ and for all primes $q_j \mid n_2$, the sequence $S$ contains at least two unit elements modulo $q_j$.

In this case, without loss of generality, let $S_1 = x_1 \cdot x_2 \cdot \ldots \cdot x_t$ be a subsequence of length $t \leq 3\omega(n_1)+ 2\omega(n_2) \leq m$ such that $S_1$ has at least three units modulo each prime $p_i$ and at least two units modulo each prime $q_j$. Now let us extend this subsequence $S_1$ to a subsequence $S_2 = x_1 \cdot x_2 \cdot \ldots \cdot x_m$ of $S$ of length $m$. Then by Lemma \ref{three} and Lemma \ref{two}, for each $p_i$ and $q_j$, we have $$\sum_{k=1}^m a_k^i x_k \equiv 0 \pmod {p_i^{e_i}} \text{ and } \sum_{k=1}^m b_k^j x_k \equiv 0 \pmod {q_j^{f_j}},$$ where $a_k^i \in T_{p_i^{e_i}}$ and $b_k^j \in T_{q_j^{f_j}}$. Now the result follows from the Chinese Remainder Theorem and Lemma \ref{crt}.
\end{proof}

\medskip

\noindent{\bf Proof of Theorem \ref{mt1}.} Since $n=n_1n_2$ is an odd integer such that $n_1 =\prod_{i =1}^r p_i$ and $n_2=\prod_{j =1}^s q_j$  with primes $p_i \equiv 1 \pmod 3$ and $q_j \equiv 2 \pmod 3$ and $7, 13 \nmid n$, we have $n \geq p_1p_2\ldots p_{\omega(n_1)}q_1q_2\ldots q_{\omega(n_2)} > 13^{\omega(n_1)}5^{\omega(n_2)} \geq 3\omega(n_1)+ 2 \omega(n_2)$. Hence by Proposition \ref{2prop1}, any sequence over $\mathbb{Z}_n$ of length $n + 2\Omega(n_1)+\Omega(n_2)$ has an $T_n$-weighted zero-sum subsequence of length $n$. This proves that $E_{T_n} (\mathbb{Z}_n) \leq n + 2\Omega(n_1)+ \Omega(n_2)$. Therefore from Proposition \ref{lowerbound2}, we get $E_{T_n} (\mathbb{Z}_n) = n + 2\Omega(n_1)+ \Omega(n_2)$, and the relation \eqref{2eq2} gives $D_{T_n}(\mathbb{Z}_n) = 2\Omega(n_1)+ \Omega(n_2) +1$.\qed

\section{Proof of Theorem \ref{extremal}}

Let $n=n_1n_2$ be a square-free odd integer which is not divisible by 7 and 13 such that $n_1 =\prod_{i =1}^r p_i$ and $n_2=\prod_{j =1}^s q_j$  with primes $p_i \equiv 1 \pmod 3$ and $q_j \equiv 2 \pmod 3$. Then we know that $D_{T_n}(\mathbb{Z}_n) = 2\Omega(n_1)+ \Omega(n_2) +1$. We now give a method to construct $T_n$-extremal sequences in $\mathbb{Z}_n$.

\begin{lemma}\label{ext1}
Let $n=n_1n_2$ be a square-free odd integer which is not divisible by 7 and 13 such that $n_1 =\prod_{i =1}^r p_i$ and $n_2=\prod_{j =1}^s q_j$  with primes $p_i \equiv 1 \pmod 3$ and $q_j \equiv 2 \pmod 3$. Let $p$ be a prime divisor of $n$, and $n^{'}=\frac{n}{p}$. Then
the following constructions give us an $T_n$-extremal sequence $S$ in $\mathbb{Z}_n$:

{\noindent \sf Case 1.} $p \equiv 1 \pmod 3$. 

Let $S_1^{'}:x_1 \cdot x_2 \cdot \ldots \cdot x_{\ell}$ be an $T_{n^{'}}$- extremal sequence in $\mathbb{Z}_{n^{'}}$, and $x^*, x^{**} \in \mathbb{Z}_n$ be such that the image of the sequence $x^* \cdot x^{**}$ under the natural map $\mathbb{Z}_n \rightarrow \mathbb{Z}_p$ has no $T_p$-weighted zero-sum subsequence. Then the sequence $S:x^* \cdot x^{**}\cdot px_1 \cdot px_2 \cdot \ldots \cdot px_{\ell}$ is $T_n$-extremal sequence in $\mathbb{Z}_n$.  

\medskip

{\noindent \sf Case 2.} $p \equiv 2 \pmod 3$.

Let $S_1^{'}:x_1 \cdot x_2 \cdot \ldots \cdot x_{\ell}$ be an $T_{n^{'}}$- extremal sequence in $\mathbb{Z}_{n^{'}}$, and $x^* \in \mathbb{Z}_n$ be such that it is not divisible by $p$. Then the sequence $S:x^* \cdot px_1 \cdot px_2 \cdot \ldots \cdot px_{\ell}$ is $T_n$-extremal sequence in $\mathbb{Z}_n$.
\end{lemma}

\begin{proof}

{\noindent \sf Case 1.} $p \equiv 1 \pmod 3$. 

Since $S_1^{'}$ is an $T_{n^{'}}$-extremal sequence in $\mathbb{Z}_{n^{'}}$ of length $\ell$ , we have $\ell= 2 \Omega(\frac{n_1}{p})+\Omega(n_2)$. Let $S$ be the sequence as in the statement, and if possible let, $T$ be a $T_n$-weighted zero-sum subsequence of $S$.

Define the sequence $S_1: px_1 \cdot px_2 \cdot \ldots \cdot px_{\ell}$. If $T$ is a subsequence of $S_1$, then we get a contradiction that $S_1^{'}$ has a non-empty $T_{n^{'}}$-weighted zero-sum subsequence, as the image of $T_n$ under the natural map $\mathbb{Z}_n \rightarrow \mathbb{Z}_{n^{'}}$ is contained in $T_{n^{'}}$.

Thus $T$ is divisible by either $x^*$ or $x^{**}$. In this case also, by taking the $T_n$-weighted sum of $T$, we get an $T_p$-weighted zero-sum subsequence of the image of $x^*\cdot x^{**}$ under the natural map $\mathbb{Z}_n \rightarrow \mathbb{Z}_p$. A contradiction to our choice of $x^*, x^{**}$.

Therefore the sequence $S$ has no non-empty $T_n$-weighted zero-sum subsequence, and since the length of $S$ is $\ell+2=2\Omega(n_1)+ \Omega(n_2)$, $S$ is an $T_n$-extremal sequence in $\mathbb{Z}_n$.

\medskip

{\noindent \sf Case 2.} $p \equiv 2 \pmod 3$. 

Since $S_1^{'}$ is an $T_{n^{'}}$-extremal sequence of length $\ell$ in $\mathbb{Z}_{n^{'}}$, we have $\ell= 2\Omega(n_1)+ \Omega(\frac{n_2}{p})$. Let $S$ be the sequence as in the statement. If possible let $T$ be a $T_n$-weighted zero-sum subsequence of $S$. 

Define the sequence $S_1: px_1 \cdot px_2 \cdot \ldots \cdot px_{\ell}$. If $T$ is a subsequence of $S_1$, then we get a contradiction that $S_1^{'}$ has a non-empty $T_{n^{'}}$-weighted zero-sum subsequence, as the image of $T_n$ under the natural map $\mathbb{Z}_n \rightarrow \mathbb{Z}_{n^{'}}$ is contained in $T_{n^{'}}$.

Thus $T$ is divisible by $x^*$. Now by taking the $T_n$-weighted sum of $T$, we get a contradiction that $p \mid x^*$. Therefore the sequence $S$ has no non-empty $T_n$-weighted zero-sum subsequence, and since the length of $S$ is $\ell+1=2\Omega(n_1)+ \Omega(n_2)$, $S$ is an $T_n$-extremal sequence in $\mathbb{Z}_n$.

\end{proof}

We now show that the construction of $T_n$-extremal sequences given in Lemma \ref{ext1} is the only way to construct it. Before that, we need the following lemma.

\begin{lemma}\label{lem}
Let $n=n_1n_2$ be a square-free odd integer which is not divisible by 7 and 13 such that $n_1 =\prod_{i =1}^r p_i$ and $n_2=\prod_{j =1}^s q_j$  with primes $p_i \equiv 1 \pmod 3$ and $q_j \equiv 2 \pmod 3$. Let $S$ be a sequence over $\mathbb{Z}_n$ of length $\ell=2\Omega(n_1)+ \Omega(n_2)$, and $q$ a prime divisor of $n$. If $q \mid n_1$ (respectively, $q \mid n_2$) and $T$ is a subsequence of $S$ of length $\ell-1$ (respectively, $\ell$) such that every term of $T$ is divisible by $q$, then $S$ has an $T_n$-weighted zero-sum subsequence.
\end{lemma}

\begin{proof}
Let $q \mid n_1$, $n_{1}^{'}=\frac{n_1}{q}$ and $n^{'}=n_1^{'}n_2$. Let $T$ be the subsequence of $S$ of length $\ell-1$ such that every term of $T$ is divisible by $q$ and $T^{'}$ the sequence in $\mathbb{Z}_{n^{'}}$ which is obtained by dividing each term of $T$ by $q$. Then the length of $T^{'}$ is $\ell-1=2\Omega(n_1^{'})+ \Omega(n_2)+1$, and hence $T^{'}$ has an $T_{n^{'}}$-weighted zero-sum subsequence in $\mathbb{Z}_{n^{'}}$. Therefore, by Remark \ref{2rmk2}, we get that $T$ has an $T_{n}$-weighted zero-sum subsequence in $\mathbb{Z}_n$.

Now, let $q \mid n_2$, $n_{2}^{'}=\frac{n_2}{q}$ and $n^{'}=n_1n_2^{'}$. Let $T$ be the subsequence of $S$ of length $\ell$ such that every term of $T$ is divisible by $q$ and $T^{'}$ the sequence in $\mathbb{Z}_{n^{'}}$ which is obtained by dividing each term of $T$ by $q$. Then the length of $T^{'}$ is $\ell=2\Omega(n_1)+ \Omega(n_2^{'})+1$, and hence $T^{'}$ has an $T_{n^{'}}$-weighted zero-sum subsequence in $\mathbb{Z}_{n^{'}}$. Therefore, by Remark \ref{2rmk2}, we get that $T$ has an $T_{n}$-weighted zero-sum subsequence in $\mathbb{Z}_n$.

\end{proof}

The previous lemma has the following obvious corollary.

\begin{corollary}\label{exact}
Let $n=n_1n_2$ be a square-free odd integer which is not divisible by 7 and 13 such that $n_1 =\prod_{i =1}^r p_i$ and $n_2=\prod_{j =1}^s q_j$  with primes $p_i \equiv 1 \pmod 3$ and $q_j \equiv 2 \pmod 3$. Let $S$ be a $T_n$-extremal sequence in $\mathbb{Z}_n$. Then for every prime divisor $q$ of $n_1$ (respectively, of $n_2$), $q$ is co-prime to at least two terms (respectively, at least one term) of $S$. 
\end{corollary}

We are now in a position to prove Theorem \ref{extremal}.

\medspace

\noindent{\bf Proof of Theorem \ref{extremal}.} Let $n=n_1n_2$ be a square-free odd integer which is not divisible by 7 and 13 such that $n_1 =\prod_{i =1}^r p_i$ and $n_2=\prod_{j =1}^s q_j$  with primes $p_i \equiv 1 \pmod 3$ and $q_j \equiv 2 \pmod 3$. From Lemma \ref{ext1} we get that the sequences given in the statement of the theorem are $T_n$-extremal sequences in $\mathbb{Z}_n$.

Conversely let $S:x_1 \cdot x_2 \cdot \ldots \cdot x_{\ell}$ be an $T_n$-extremal sequence in $\mathbb{Z}_n$. If for every prime divisor $p$ of $n_1$, at least three term of $S$ are co-prime with $p$ and for every prime divisor $p$ of $n_2$, at least two term of $S$ are co-prime with $p$, then by Lemma \ref{three}, Lemma \ref{two} and Observation \ref{obs} we get a contradiction that $S$ is an $T_n$-weighted zero-sum sequence.

Thus there is a prime divisor $p$ of $n_1$ such that at most two term of $S$ are co-prime to $p$, or there is a prime divisor $p$ of $n_2$ such that at most one term of $S$ is co-prime to $p$.

{\noindent \sf Case 1.}

Let $p$ be a prime divisor of $n_1$ such that at most two terms of $S$ are co-prime to $p$. By Corollary \ref{exact}, exactly two term of $S$ are co-prime to $p$, say $x_1$ and $x_2$. Let $n^{'}=\frac{n}{p}$ and $x_i=py_i$ for all $3 \leq i \leq \ell$. Then the sequence $S^{'}: y_3 \cdot \ldots \cdot y_{\ell}$ of length $\ell-2=D_{T_{n^{'}}}(\mathbb{Z}_{n^{'}})-1$ is an $T_{n^{'}}$-extremal sequence in $\mathbb{Z}_{n^{'}}$. Since if $S^{'}$ has an $T_{n^{'}}$-weighted zero-sum subsequence, then similarly as in the proof of Lemma \ref{lem}, $S$ has an $T_{n}$-weighted zero-sum subsequence, a contradiction.

Also, the image of the sequence $x_1 \cdot x_2$ under the natural map $\mathbb{Z}_n \rightarrow \mathbb{Z}_p$ is $T_p$-extremal. Because if it is not extremal, then there exist $a, b \in T_n$ such that $a x_1+bx_2=0$ in $\mathbb{Z}_p$, and every term of the sequence $a x_1+bx_2 \cdot x_3 \cdot \ldots \cdot x_{\ell}$ is divisible by $p$. Therefore $S$ has an $T_n$-weighted zero-sum subsequence, a contradiction.

{\noindent \sf Case 2.}

Let $p$ be a prime divisor of $n_2$ such that at most one term of $S$ is co-prime to $p$. By Corollary \ref{exact}, exactly one term of $S$ is co-prime to $p$, say $x_1$. Let $n^{'}=\frac{n}{p}$ and $x_i=py_i$ for all $2 \leq i \leq \ell$. Then the sequence $S^{'}: y_2 \cdot \ldots \cdot y_{\ell}$ of length $\ell-1=D_{T_{n^{'}}}(\mathbb{Z}_{n^{'}})-1$ is an $T_{n^{'}}$-extremal sequence in $\mathbb{Z}_{n^{'}}$. Since if $S^{'}$ has an $T_{n^{'}}$-weighted zero-sum subsequence, then similarly as in the proof of Lemma \ref{lem}, $S$ has an $T_{n}$-weighted zero-sum subsequence, a contradiction.

This proves that the sequence $S$ is equivalent to the sequences given in the statement of the theorem. 
\qed

\medskip

\end{document}